\documentclass[11pt,a4paper]{amsart}
\usepackage{mypackage}

\date{}

\author{Nicolas Tholozan}
\thanks{}

\address{University of Luxembourg \\
Campus Kirchberg, Mathematics Research Unit, BLG \\
6, rue Richard Coudenhove-Kalergi \\
L-1359 Luxembourg
}
\email{nicolas.tholozan@uni.lu}

\title[Hilbert metrics and Hitchin representations in $\PSL(3,\R)$]{Entropy of Hilbert metrics and length spectrum of Hitchin representations in $\PSL(3,\R)$}

\begin{document}

\maketitle

\begin{abstract}
This article studies the geometry of proper convex domains of the projective space $\ProjR{n}$. These convex domains carry several projective invariant distances, among which the Hilbert distance $d^H$ and the Blaschke distance $d^B$. We prove a thin inequality between those distances: for any two points $x$ and $y$ in a proper convex domain,
\[d^B(x,y) < d^H(x,y) +1~.\]

We then give two interesting consequences. The first one is a conjecture of Colbois and Verovic on the volume entropy of Hilbert geometries: for any proper convex domain of $\ProjR{n}$, the volume of a ball of radius $R$ grows at most like $e^{(n-1)R}$.

 The second consequence is the following fact: for any Hitchin representation $\rho$ of a surface group into $\PSL(3,\R)$, there exists a Fuchsian representation $j$ in $\PSL(2,\R)$ such that the length spectrum of $j$ is uniformly smaller than the length spectrum of $\rho$. This answers positively (for $n=3$) to a conjecture of Lee and Zhang.
\end{abstract}

\tableofcontents

\section*{Introduction}

An open domain of $\ProjR{n}$ is \emph{convex} if its intersection with any projective line is connected. It is called \emph{proper} if its closure does not contain a projective line.

The geometry of proper convex open domains has been extensively studied since Hilbert introduced them as examples of metric spaces ``whose geodesics are straight lines'' \cite{Hilbert95}. More precisely, Hilbert provided any proper convex open domain with a natural Finsler metric for which projective segments are geodesics. Moreover, this metric is a projective invariant and therefore any projective transformation preserving a convex acts isometrically for the Hilbert metric. When the convex domain is a ball, we recover the Klein model of hyperbolic space.

The Hilbert metric may be the most natural metric on a convex, but it is not the easiest to study. It is (almost) never Riemannian and in many interesting cases it is not $\mathcal{C}^2$. Another ``natural'' (i.e. projectively invariant) choice of a metric is the \emph{Blaschke metric} (also known as \emph{affine metric}) that arises in the theory of affine spheres developed by Blaschke, Calabi, Cheng and Yau. The definition of the Blaschke metric relies on a deep analytic theorem and may seem difficult to apprehend at first. The counterpart is that it is Riemannian, smooth and has nice curvature properties (see theorem \ref{t:RicciBlaschke}).

One can hope in general that the Blaschke metric is ``close enough'' to the Hilbert metric so that we can deduce, from good analytic properties of the Blaschke metric, similar properties for the ``wilder'' Hilbert metric.

\subsection{A comparison lemma}

Let us fix a proper open convex domain $\Omega$ in $\ProjR{n}$. Denote by $h^H_\Omega$ (resp. $h^B_\Omega$) its Hilbert (resp. Blaschke) metric and by $d^H_\Omega$ (resp. $d^B_\Omega$) the associated distance. (Very often we will omit to index those objects by $\Omega$.)
In a recent paper, Benoist and Hulin proved that the Blaschke and Hilbert metrics are uniformly comparable.
\begin{CiteThm}[Benoist--Hulin, \cite{BenoistHulin13}] \label{t:BenoistHulin}
There exists a positive constant $C_n$ (depending only on the dimension), such that 
\[ \frac{1}{C_n} h^H \leq h^B \leq C_n h^H~. \]
\end{CiteThm}

The central result of this paper is a refinement of the right inequality:
\begin{MonLem} \label{l:ComparaisonHilbertBlaschke}
For any $x,y \in \Omega$, 
\[ d^B(x,y) < d^H(x,y) + 1~. \]
\end{MonLem}

\begin{rmk}
Clearly, this lemma only refines Benoist--Hulin's theorem when $d^H(x,y)$ is big enough. Note that our proof will make use of Benoist--Hulin's theorem.
\end{rmk}

\begin{rmk}
One could hope for a stronger inequality, namely that $d^B \leq d^H$. However, computing both metrics when $\Omega$ is a square in $\ProjR{2}$ shows that this stronger inequality does not always hold.
\end{rmk}

We will now give two important consequences of lemma \ref{l:ComparaisonHilbertBlaschke}.

\subsection{Volume entropy of convex domains}

There is no completely standard way to associate a volume form to the Hilbert metric, but there is a natural class of volume forms.

We call a volume form $\vol$ on $\Omega$ \emph{uniform} if there exists a constant $K>1$ such that for any point $x \in \Omega$,
\[\frac{1}{K} \leq \vol\left( \{u \in T_x\Omega \mid h^H(u) \leq 1\}\right) \leq K~.\]
Note that, according to Benoist--Hulin's theorem, and example of such a form is the volume form canonically associated to the Blaschke metric.

Denote by $B^H(x,R)$ the ball of center $x$ and radius $R$ for the Hilbert metric.
\begin{definition}
The \emph{volume entropy} of the Hilbert metric is defined by
\[\Ent(h^H) = \limsup_{R\to +\infty} \frac{1}{R} \log \Vol\left( B^H(o,R) \right)~,\]
where $o$ is some base point in $\Omega$ and the volume $\Vol$ is computed with respect to a regular volume form $\vol$.
\end{definition}
It is not difficult to see that this entropy does not depend on the precise choice of a volume form, nor on the base point $o$. One can define in the same way the volume entropy of the Blaschke metric by replacing $B^H(o,R)$ by $B^B(o,R)$, the ball with respect to the Blaschke metric. The volume entropy of the Hilbert metric is sometimes called the \emph{volume entropy of $\Omega$}, but for our purpose it is better to distinguish between the entropies of  the Blaschke and Hilbert metrics. \\

It is a well-known conjecture in Hilbert geometry that the entropy of the Hilbert metric of a properly convex domain in $\ProjR{n}$ is bounded above by $n-1$. It seems to date back to the work of Colbois and Verovic  \cite{ColboisVerovic04}, who proved that, if the boundary of $\Omega$ is sufficiently regular, then the entropy of $\Omega$ is actually equal to $n-1$. This was later refined by Berck, Bernig and Vernicos in \cite{BBV10}, where they also prove the conjecture in dimension $2$. Vernicos then recently proved the conjecture in dimension $3$ \cite{Vernicos14}.

In another direction, Crampon \cite{Crampon09} proved the conjecture in any dimension, assuming $\Omega$ is \emph{divisible} and \emph{hyperbolic} (i.e. preserved by a discrete Gromov-hyperbolic group $\Gamma$ acting cocompactly). In that case, the volume entropy of the Hilbert metric can be interpreted as the dynamical entropy of the geodesic flow on $\Omega/\Gamma$.\\

Here we prove the conjecture in full generality:
\begin{MonThm} \label{t:VolumeEntropy}
Let $\Omega$ be a proper convex open domain of $\ProjR{n}$. Then the volume entropy of the Hilbert metric on $\Omega$ satisfies
\[\Ent \left(h^H_\Omega\right) \leq n-1~.\]
\end{MonThm}

As we said, the main ingredient in the proof of this theorem is lemma \ref{l:ComparaisonHilbertBlaschke}. Indeed, this lemma implies the following inequality (lemma \ref{l:EntropyHilbertBlaschke}):
\[\Ent\left(h^H\right) \leq \Ent \left(h^B\right)~.\]
It remains to prove that $\Ent \left(h^B\right) \leq n-1$, which is a consequence of a famous theorem of Bishop for the volume entropy of Riemannian metrics (theorem \ref{t:Bishop}), together with a theorem of Calabi giving a lower bound for the Ricci curvature of the Blaschke metric:

\begin{CiteThm}[Calabi, \cite{Calabi72}] \label{t:RicciBlaschke}
The Ricci curvature of the Blaschke metric satisfies
\[\Ricci(h^B) \geq -(n-1) h^B~.\]
\end{CiteThm}

\subsection{Length spectrum of Hitchin representations in $\PSL(3,\R)$}

Let $g$ be an isometry of some metric space $(X,d)$. We define the \emph{translation length} of $g$ as the number
\[l(g) = \lim_{n\to +\infty} \frac{1}{n} d(x,g^n\cdot x)~,\]
where $x$ is any point in $X$.\\

Let us now consider $S$ a closed connected oriented surface of genus greater than $1$ and denote by $\Gamma$ its fundamental group. A representation $\rho: \Gamma \to \PSL(3,\R)$ is called a \emph{Hitchin representation} if $\rho$ is injective and $\rho(\Gamma)$ divides a (necessarily unique) open convex domain $\Omega_\rho$ in $\ProjR{2}$, i.e. $\rho(\Gamma)$ acts properly discontinuously and cocompactly on $\Omega_\rho$. (The terminology of ``Hitchin representation'' will be explained in the next paragraph.)

\begin{definition}
The \emph{length spectrum} of a representation $\rho: \Gamma \to \Isom(X,d)$ is the function $L_\rho$ that associates to (the conjugacy class of) an element $\gamma \in \Gamma$ the translation length of $\rho(\gamma)$.
\end{definition}
In particular, one can define the length spectrum of a Hitchin representation $\rho$ in $\PSL(3,\R)$ by seeing $\rho$ as a representation into the isometry group of $(\Omega_\rho, h^H)$. \\

Denote by $\H^2$ the hyperbolic plane. Recall that a representation $j: \Gamma \to \PSL(2,\R) \simeq \Isom^+(\H^2)$ is \emph{Fuchsian} if it is injective and acts properly discontinuously on $\H^2$. Note that, identifying $\PSL(2,\R)$ with $\SO_0(2,1)$, one can see Fuchsian representations as special cases of Hitchin representations in $\PSL(3,\R)$ dividing a disc in $\ProjR{2}$.

Motivated by questions arising in anti-de Sitter geometry, the author recently proved with Bertrand Deroin \cite{DeroinTholozan} a strong ``domination'' result for certain representations of a surface group.

\begin{CiteThm}[Deroin--T.]
Let $\rho$ be a representation of $\Gamma$ into the isometry group of a complete, simply connected Riemannian manifold of sectional curvature bounded above by $-1$. Then there exists a Fuchsian representation $j$ such that
\[L_j \geq L_\rho~.\]
\end{CiteThm}

Moreover, the inequality in this theorem can be made strict unless $\rho$ itself is ``Fuchsian'' in some very rigid sense.
This applies mostly to representations into Lie groups of rank $1$ (seen as isometry groups of their symmetric spaces). In the case of $\PSL(2,\C)$ for instance, it gives a new proof of the famous rigidity theorem of Bowen \cite{Bowen79} for the entropy of quasi-Fuchsian representations.

Lemma \ref{l:ComparaisonHilbertBlaschke} allows us to prove a similar result (but with reverse inequality) for Hitchin representations in $\PSL(3,\R)$.

\begin{MonThm} \label{t:MinorationHitchin}
Let $\rho$ be a Hitchin representation into $\PSL(3,\R)$. Then either $\rho$ is Fuchsian or there exists a constant $K > 1$ and a Fuchsian representation $j$ such that
\[L_\rho \geq K L_j~.\]
\end{MonThm}
(Here we say that $\rho$ is Fuchsian if the convex $\Omega_\rho$ is an ellipsoid.)\\

Again, the proof of theorem \ref{t:MinorationHitchin} will use the Blaschke metric as an intermediate comparison. One starts by deducing from lemma \ref{l:ComparaisonHilbertBlaschke} that the length spectrum of $\rho$ with respect to the Blaschke metric is uniformly smaller than the length spectrum of $\rho$ with respect to the Hilbert metric (corollary \ref{c:TranslationHilbertBlaschke}). Then one considers $h^P$ the unique complete metric on $\Omega_\rho$ conformal to the Blaschke metric and of curvature $-1$. Applying Calabi's theorem together with the classical Ahlfors--Schwarz--Pick lemma (lemma \ref{l:AhlforsSchwarzPick}), we obtain that 
\[h^B \geq h^P~.\]
The action of $\rho$ on $(\Omega_\rho,h^P)$ is thus isometrically conjugated to a Fuchsian representation $j$ acting on $\H^2$ whose length spectrum will satisfy
\[L_j \leq L_\rho~.\]

\subsection{Hitchin Representations in higher dimension}

Theorem \ref{t:MinorationHitchin} gives an answer, for $n=3$, to a more general question about representations of a surface group in $\PSL(n,\R)$. We finish this introduction by mentionning this possible generalization in order to motivate theorem \ref{t:MinorationHitchin}.\\

We still denote by $S$ a connected closed oriented surface of genus at least $2$ and by $\Gamma$ its fundamental group. A representation of $\Gamma$ in $\PSL(2,\R)$ induces in a natural way a representation in $\PSL(n,\R)$, simply by post-composing with the irreducible representation
\[\iota_n: \PSL(2,\R) \to \PSL(n,\R)~.\]
By extension, we will say that a representation of $\Gamma$ into $\PSL(n,\R)$ is Fuchsian if it is of the form $\iota_n \circ j$ with $j$ a Fuchsian representation in  $\PSL(2,\R)$.

In \cite{Hitchin92}, Hitchin described the connected components of the space of representations into $\PSL(n,\R)$ containing Fuchsian representations. For this reason, representations of $\Gamma$ into $\PSL(n,\R)$ that can be continuously deformed into Fuchsian representations are called \emph{Hitchin representations}.

In the special case where $n=3$, Choi and Goldman \cite{ChoiGoldman93} proved that Hitchin representations are exactly those dividing a convex set of $\ProjR{2}$:
\begin{CiteThm}[Choi--Goldman, \cite{ChoiGoldman93}]
A representation of the fundamental group of a connected closed oriented surface into $\PSL(3,\R)$ is Hitchin if and only if it divides a convex set of $\ProjR{2}$.
\end{CiteThm}
This explains the terminology we used in the previous paragraph.\\

One can define the length spectrum of a representation in $\PSL(n,\R)$ by looking at the action of $\PSL(n,\R)$ on its symmetric space $\PSL(n,\R)/ \PSO(n)$. This symmetric space carries several $\PSL(n,\R)$-invariant Finsler metrics (all of which are bi-Lipschitz equivalent to the symmetric Riemannian metric). An interesting choice is to provide the symmetric space with the unique $\PSL(n,\R)$-invariant Finsler metric such that, for any $\lambda_1 > \ldots > \lambda_n$ with $\sum_i \lambda_i = 0$, the diagonal matrix
\[\left( 
\begin{matrix}
e^{\lambda_1} & \ & \ \\
\ & \ddots & \ \\
\ & \ & e^{\lambda_n}
\end{matrix} \right)\]
has translation length $\frac{1}{2}(\lambda_1 - \lambda_n)$. We will denote by $L_\rho$ the length spectrum of a representation $\rho$ with respect to this particular Finsler metric. In the case of Hitchin representations in $\PSL(3,\R)$, this length spectrum turns out to be exactly the length spectrum of $\rho$ as we defined it before (i.e. by looking at the action on $\Omega_\rho$).\\

%One cannot hope to generalize the result of \cite{DeroinTholozan} to Hitchin representations. On the contrary, several facts show that Fuchsian representations are in some sense ``minimal'' among Hitchin representations. One of them is the rigidity result of Crampon \cite{Crampon09} concerning the entropy of divisible convex sets, which applies in particular to Hitchin representations in $\PSL(3,\R)$ and was recently extended to Hitchin reprentations in higher dimension by Potrie--Sambarino. To state this theorem let us recall first the definition of the \emph{entropy} of a Hitchin representation. \\

%\begin{definition}
%The \emph{entropy} of a Hitchin representation $\rho: \Gamma \to \PSL(n,\R)$ is the number
%\[\mathcal{H}(\rho) = \limsup_{R \to + \infty} \frac{1}{R} \log \left( \sharp \left\{[\gamma] \in [\Gamma] %\mid L_\rho(\gamma) \leq R \right\} \right)~,\]
%where $[\Gamma]$ denotes the set of conjugacy classes of elements in $\Gamma$.
%\end{definition}

%\begin{CiteThm}[Crampon \cite{Crampon09}, Potrie--Sambarino \cite{PotrieSambarino14}]
%The entropy of a Hitchin representation into $\PSL(n,\R)$ is less or equal to $\frac{2}{n-1}$, with equality if an only if the representation is Fuchsian.
%\end{CiteThm}

In a recent work, Lee and Zhang \cite{LeeZhang14} prove that Hitchin representations satisfy the following property:

\begin{CiteThm}[Lee, Zhang, \cite{LeeZhang14}]
There exists a constant $C$ such that if $\gamma$, $\gamma'$ are two curves on $S$ that are not homotopic to disjoint curves, then, for any Hitchin representation $\rho$ in $\PSL(n,\R)$, one has
\[\left( \exp(L_\rho(\gamma)) - 1 \right) \left( \exp(L_\rho(\gamma')/(n-1)) -1 \right) > 1~.\]
\end{CiteThm}

This result is a slightly weaker version of the classical \emph{collar lemma} for Fuchsian representations. In that same paper, Lee and Zhang conjecture that a sharper version of the collar lemma for Hitchin representations should hold as a consequence of the following more general fact:

\begin{conjecture}[Lee--Zhang]
for any Hitchin representation $\rho: \Gamma \to \PSL(n,\R)$, there is a Fuchsian representation $j$ such that
\[L_j \leq L_\rho~.\]
\end{conjecture}

Theorem \ref{t:MinorationHitchin} answers positively to this conjecture when $n=3$. As pointed out by Lee and Zhang, we obtain as a corollary:
\begin{MonCoro}[Collar lemma for Hitchin representations in $\PSL(3,\R)$]
Let $\gamma$, $\gamma'$ be two curves on $S$ that are not homotopic to disjoint curves. Then, for any Hitchin representation $\rho$ in $\PSL(3,\R)$, one has
\[\sinh(L_\rho(\gamma)/2) \cdot \sinh(L_\rho(\gamma')/2) > 1~.\]
\end{MonCoro}

Labourie explained to us that the Lee--Zhang conjecture cannot hold anymore for $n\geq 4$, as a consequence of several recent works on Hitchin representations (\cite{BCLS13}, \cite{PotrieSambarino14}). The contradiction comes from Hitchin representations in $\PSp(2k,\R)$ and $\PSO(k,k+1)$. This leads us to modify the conjecture of Lee and Zhang:
\begin{conjecture}{\ \\}
\begin{itemize}
\item For any Hitchin representation $\rho$ in $\PSL(2k,\R)$, there is a Hitchin representation $j$ in $\PSp(2k,\R)$ such that
\[L_j \leq L_\rho~,\]
\item for any Hitchin representation $\rho$ in $\PSL(2k+1,\R)$, there is a Hitchin representation $j$ in $\PSO(k,k+1)$ such that
\[L_j \leq L_\rho~.\]
\end{itemize}
\end{conjecture}
Note that the irreducible representation of $\PSL(2,\R)$ into $\PSL(3,\R)$ identifies $\PSL(2,\R)$ with $\PSO(2,1)$. The conjecture would thus be a generalization of theorem \ref{t:MinorationHitchin}.

\subsection{Content of the article}
In section \ref{s:Rappels}, we recall the definitions of the Blaschke and Hilbert metrics. We prove lemma \ref{l:ComparaisonHilbertBlaschke} and theorem \ref{t:VolumeEntropy} in section~\ref{s:ComparaisonBlaschkeHilbert}. In section \ref{s:HitchinRepresentations}, we focus on representations of surface groups in $\PSL(3,\R)$. We prove theorem \ref{t:MinorationHitchin} and make several remarks concerning the behaviour of the length spectrum of Hitchin representations that are ``far from being Fuchsian''. It was brought to our attention by Courtois that these remarks are essentially contained in a recent paper by Xin Nie \cite{Nie15}.

\subsection{Acknowledgements}
The author is thankful to Yves Benoist for enlightening discussions, to Gilles Courtois for pointing out the paper of Xin Nie, and to Constantin Vernicos for his careful reading and thoughtful remarks on the first version of this paper.

\section{Hilbert and Blaschke metrics} \label{s:Rappels}

\subsection{The Hilbert metric} \label{ss:HilbertMetric}
Fix an open proper convex domain $\Omega$ in $\ProjR{n}$. Recall that the \emph{cross-ratio} of $4$ points $x_1,x_2,x_3,x_4$ in a projective line is the number $t = [x_1,x_2,x_3,x_4]$ such that $(x_1,x_2,x_3,x_4)$ is sent to $(0,1,\infty,t)$ in some affine chart.

Given any two points $x$ and $y$ in $\Omega$, let $a$ and $b$ be the two intersections of the projective line going through $x$ and $y$ with the boundary of $\Omega$ (so that $x$ is between $a$ and $y$). 

\begin{definition}
The Hilbert distance between $x$ and $y$ is defined by
\[d^H_\Omega(x,y) = \frac{1}{2}\log | [a,x,b,y]|~.\]
\end{definition}
It is classical (though non trivial) that this indeed defines a distance on~$\Omega$. Moreover, this distance is infinitesimally generated by a Finsler metric $h^H$. More precisely, let us define the Hilbert metric by
\[h^H_{\Omega,x}(u) = \frac{\d^2}{\d t^2}_{|t= 0}{d^H_\Omega(x,x+tu)^2}\]
for a vector $u$ tangent to a point $x$ in $\Omega$.
\begin{proposition}
The function $h^H_{\Omega,x}$ is the square of a norm on $T_x\Omega$ and for all $x, y \in \Omega$, we have
\[d^H_\Omega(x,y) = \inf_\gamma \int_0^1 \sqrt{h^H_\Omega(\gamma'(t))}~,\]
where the infimum is taken over all $\mathcal{C}^1$ paths from $x$ to $y$.
\end{proposition}

The Hilbert metric is a projective invariant: if $\Omega$ is a proper open convex domain and $g$ a projective transformation, then the Hilbert metric of $g(\Omega)$ is $g_*h^H_\Omega$. In particular, every projective transformation $g$ such that $g(\Omega) = \Omega$ is an isometry for $(\Omega,h^H_\Omega)$. When $\Omega$ is an ellipsoid, it is preserved by a subgroup of $\PSL(n+1,\R)$ conjugated to $\PSO(n,1)$. This implies that $h^H_\Omega$ is Riemannian and $(\Omega,h^H_\Omega)$ is isometric to the hyperbolic space $\H^n$. \\

From now on we will omit to index the Hilbert metric by $\Omega$. Hopefully it will not lead to any confusion.

The metric $h^H$ at a point ``sees'' in some sense the shape of the boundary of $\Omega$. This implies that the Finsler metric is never Riemannian unless $\Omega$ is an ellipsoid, and often has low regularity. In many intersesting examples it is therefore almost impossible to use analytic tools to study the Hilbert metric. This motivates the introduction of an auxiliary Riemannian metric which is also projectively invariant and has better analytic properties.\\

\subsection{The Blaschke metric}

The price to pay is a construction which is much more elaborate and relies on an existence theorem for solutions of certain Monge--Amp\`ere equations. The Blaschke metric is the second fundamental form of a certain smooth hypersurface in $\R^{n+1}$ asymptotic to the cone over $\Omega$, called the \emph{affine sphere}. The definition we give here is not the most general. It is adapted from \cite{BenoistHulin13}, definition 2.1.

Denote by $L$ the restriction to $\Omega$ of the tautological $\R$-line bundle over $\ProjR{n}$ and $\xi$ a smooth and nowhere vanishing section of $L$. One can see $\xi(\Omega)$ as a smooth hypersurface in $\R^{n+1}$ transverse to the radial vector field. Let $E$ denote the pull-back of the tangent bundle of $\R^{n+1}$ by $\xi$. Then $E$ splits as a direct sum
\[E = T \Omega \oplus L~,\]
where $T\Omega$ is identified with its image by $\d \xi$ (see figure \ref{fig:AffineSphere}).

Now, the bundle $E$ inherits from $\R^{n+1}$ a volume form $\omega$ and a flat connection $\nabla$. For any vector fields $X$ and $Y$ on $\Omega$, on can write
\[\nabla_X Y = \nabla^\xi_X Y + h(X,Y) \xi~,\]
where $\nabla^\xi$ is a connection on $T\Omega$ and $h$ is a symmetric bilinear form on $\Omega$.

We say that the hypersurface $\xi(\Omega)$ is \emph{strictly convex} if $h$ is positive definite and \emph{proper} if the map $\xi: \Omega \to \R^{n+1}$ is proper.

Assume that $\xi(\Omega)$ is strictly convex. One can extend the metric $h$ on $T\Omega$ to a metric $\hat{h}$ on $E$ by deciding that the decomposition $E = T \Omega \oplus L$ is orthogonal and that $\xi(x) \in L_x$ is of norm $1$.

\begin{definition}
The hypersurface $\xi(\Omega)$ is a hyperbolic affine sphere asymptotic to $\Omega$ if $\xi(\Omega)$ is proper, convex, and \[\omega (\hat{h}) = 1~,\]
where $\omega(\hat{h})$ is the volume with respect to $\omega$ of an oriented orthonormal basis for $\hat{h}$.
\end{definition}

Finding affine spheres boils down to solving some Monge-Amp\`ere equation on $\Omega$ with boundary conditions. This allowed Cheng and Yau to prove the following:

\begin{CiteThm}[Cheng--Yau, \cite{ChengYau77}]
For any proper convex domain $\Omega$ in $\ProjR{n}$, there is a unique hyperbolic affine sphere in $(\R^{n+1}, \omega)$ asymptotic to $\Omega$.
\end{CiteThm}
If $\xi(\Omega)$ is this unique affine sphere, then the metric $h$ on $T\Omega$ is called the \emph{Blaschke metric}. We denote it $h^B$. The theorem of Cheng--Yau includes a regularity result showing that the affine sphere and the Blaschke metric are analytic.

\begin{rmk}
This affine sphere depends on the choice of a volume on $\R^{n+1}$ in a very simple way. If $\omega$ is multiplied by $\lambda$ then the affine sphere is transformed by an homothety of ratio $\frac{1}{\lambda^{n+1}}$ and the Blaschke metric is unchanged.
\end{rmk}

\begin{proposition}
When the convex $\Omega$ is an ellipsoid, the affine sphere is a hyperboloid and the Blaschke metric coincides with the Hilbert metric (see figure \ref{fig:AffineSphere}).
\end{proposition}

\begin{figure} 
\begin{center}
\includegraphics[width=12cm]{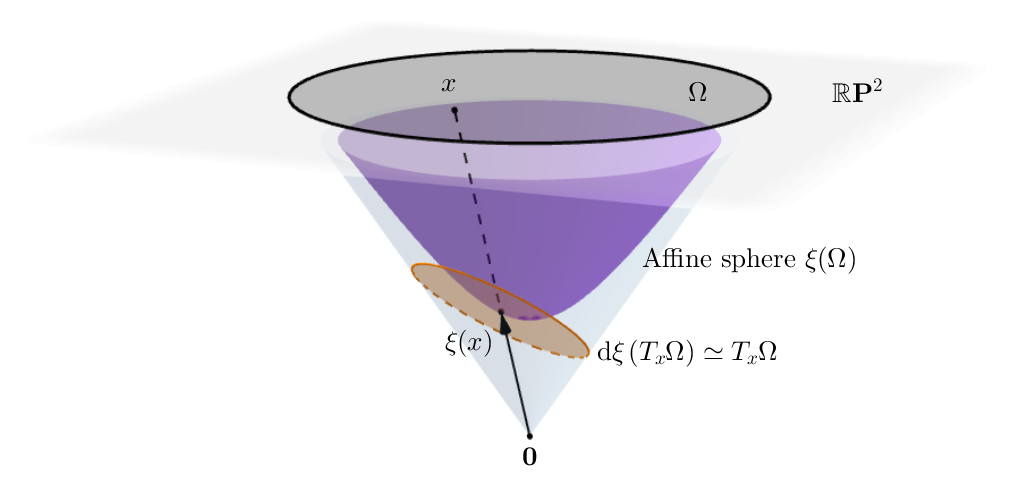}
\caption{When $\Omega$ is a disc in $\ProjR{2}$, the affine sphere is a hyperboloid.}
\label{fig:AffineSphere}
\end{center}
\end{figure}

\subsubsection{A word on the dimension $2$ case}

This paragraph will only be relevant for the remarks of section \ref{ss:AsymptoticSpectrum}.

Let $\Omega$ be a proper convex open domain in $\ProjR{n}$ and $\xi(\Omega)$ the hyperbolic affine sphere asymptotic to $\Omega$. With  the notations above, the connection $\nabla^\xi$ does not preserve $h^B$ in general. The difference between $\nabla^\xi$ and the Levi--Civita connection of $h^B$ is called the \emph{Pick tensor} of $\Omega$. Denote it by $P$, and define
\[A(X,Y,Z) = h^B(P(X,Y), Z)~.\]
Then the tensor $A$ is symmetric in $X,Y,Z$. 

\begin{proposition}[see \cite{BenoistHulin13}, lemma 4.8]
Let $\Omega$ be a proper convex domain of $\ProjR{2}$. Provide $\Omega$ with the conformal structure induced by $h^B$. Then the tensor $A$ is the real part of a holomorphic cubic differential, called the Pick form of $\Omega$.
\end{proposition}

Now, let $S$ be a closed surface, $\Gamma$ its fundamental group, and $\rho$ a Hitchin representation of $\Gamma$ in $\PSL(3,\R)$. Let $\Omega_\rho \subset \ProjR{2}$ be the convex divided by $\rho$. Then the Pick form and the conformal class of the Blaschke metric are preserved by $\rho(\Gamma)$ and induce a complex structure and a holomorphic cubic differential on $\Omega_\rho/\rho(\Gamma) \simeq S$. Loftin and Labourie proved that this actually gives a parametrization of the space of Hitchin representations of $\Gamma$ into $\PSL(3,\R)$.

\begin{CiteThm}[Labourie \cite{Labourie07}, Loftin \cite{Loftin01}] \label{t:LabourieLoftin}
Let $J$ be a conformal structure on $S$ and $\Phi$ a holomorphic cubic differential on $(S,J)$. Then there is, up to conjugacy, a unique Hitchin representation $\rho = \rho(J,\Phi)$ in $\PSL(3,\R)$ for which there exists a $\rho$-equivariant homeomorphism
\[f: \tilde{S} \to \Omega_\rho\]
sending $J$ to the conformal class of the Blaschke metric and $\Phi$ to the Pick form of $\Omega_\rho$.
\end{CiteThm}

\begin{rmk}
Generalizing Labourie--Loftin's theorem, Benoist and Hulin recently gave a parametrization of all proper convex open domains in $\ProjR{2}$ that are Gromov-hyperbolic \cite{BenoistHulin14}. Dumas and Wolf also gave a similar parametrization for convex polygons in $\ProjR{2}$ \cite{DumasWolf14}.
\end{rmk}

In section \ref{ss:AsymptoticSpectrum}, we will describe the asymptotic behaviour of the length spectrum of $\rho(J,\Phi)$ when $J$ is fixed and $\Phi$ goes to infinity.

\section{Comparison between the Hilbert and Blaschke metrics} \label{s:ComparaisonBlaschkeHilbert}

In this section, we prove lemma \ref{l:ComparaisonHilbertBlaschke} and obtain theorem \ref{t:VolumeEntropy} as a corollary.

\subsection{Proof of the main lemma}

Let $\Omega$ be a proper convex open set in $\ProjR{n}$. If $u$ is a vector in $\R^{n+1}$, we denote by $[u]$ its projection in $\ProjR{n}$.\\

Let $\mathbf{e}_1$ and $\mathbf{e}_2$ be two non zero vectors in $\R^{n+1}$ such that the corresponding points $[\mathbf{e}_1]$ and $[\mathbf{e}_2]$ in $\ProjR{n}$ are distinct points of $\partial \Omega$. We now restrict to the plane generated by those directions. 

We parametrize the projective segment between $[\mathbf{e}_1]$ and $[\mathbf{e}_2]$ by
\[ \left\{ [e^t \mathbf{e}_1 + e^{-t} \mathbf{e}_2], t\in \R\right\}~.\]
The affine sphere restricted to this segment is thus parametrized by
\[ \left\{ u(t)=e^{ t + \alpha(t)} \mathbf{e}_1 + e^{- t+ \alpha(t)} \mathbf{e}_2, t\in \R\right\}~,\]
for some function $\alpha : \R \to \R$. \\

One easily verifies that 
\[d^H([u(t_1)], [u(t_2)]) = |t_1-t_2|~,\]
and in particular,
\[h^H\left(\dt{[u(t)]}\right) = 1~.\]
In other words, $[u(t)]$ is a geodesic for the Hilbert metric.
Let us now evaluate the Blaschke metric on $\dt{[u(t)]}$.

\begin{lemme} \label{l:BlaschkeDerivative}
\[h^B\left(\dt{[u(t)]}, \dt{[u(t)]}\right) = \alpha''(t) - \alpha'^2(t) + 1~.\]
\end{lemme}

\begin{proof}
by definition of the Blaschke metric, one has
\[u''(t) = h^B\left(\dt{[u(t)]}, \dt{[u(t)]}\right) u(t) + \beta(t) u'(t)\]
for some function $\beta$. We just have to compute the coordinates of the first and second derivative of $u(t)$. We find
\[u'(t) = (\alpha'(t) + 1)e^{t + \alpha(t)}\mathbf{e}_1 + (\alpha'(t) - 1)e^{-t + \alpha(t)}\mathbf{e}_2~,\]
and
\[u''(t) = \left(\alpha''(t) + (\alpha'(t)+1)^2\right) e^{ t + \alpha(t)} \mathbf{e}_1 + 
\left(\alpha''(t) + (\alpha'(t)-1)^2\right) e^{-t + \alpha(t)} \mathbf{e}_2~.\]
To obtain the coordinates of $u''(t)$ in the basis $(u(t), u'(t))$, we invert the matrix
\[A(t) = \left( \begin{matrix}
e^{ t + \alpha(t)} & (\alpha'(t) + 1)e^{t + \alpha(t)} \\
e^{-t + \alpha(t)} & (\alpha'(t) -1)e^{-t + \alpha(t)} \end{matrix}
\right)~.
\]
We get
\[A^{-1}(t) = \frac{1}{-2e^{2 \alpha(t)}}\left(\begin{matrix}
 (\alpha'(t) -1)e^{-t + \alpha(t)} & -(\alpha'(t) + 1)e^{t + \alpha(t)} \\
-e^{-t + \alpha(t)} & e^{t + \alpha(t)}

\end{matrix}
\right)~.\]

The coordinates of $u''$ in the basis $(u,u')$ are thus given by
\[A^{-1}(t)
\left( \begin{matrix}
\left(\alpha''(t) + (\alpha'(t)+1)^2\right) e^{t + \alpha(t)} \\
\left(\alpha''(t) + (\alpha'(t)-1)^2\right) e^{-t + \alpha(t)}
\end{matrix} \right) \]
and in particular, be obtain
\begin{eqnarray*}
h^B\left(\dt{[u(t)]}, \dt{[u(t)]}\right) & = & \frac{-1}{2} \left[(\alpha'(t) -1)\left(\alpha''(t) + (\alpha'(t)+1)^2\right) \right.\\
\ & \ & - \left. (\alpha'(t)+1)\left(\alpha''(t) + (\alpha'(t)-1)^2\right) \right] \\
\ & = & \alpha''(t) - \alpha'^2(t) + 1~.
\end{eqnarray*}

\end{proof}

Now, Benoist--Hulin's theorem allows us to prove that $\alpha'$ is bounded by~$1$. More precisely, we show that

\begin{proposition} \label{p:ControleAlpha}
If $h^B \geq \frac{1}{C} h^H$ for some $C\geq 1$, then
\[\alpha'^2 \leq 1-\frac{1}{C}~.\]
\end{proposition}

\begin{proof}
Since $h^H\left(\dt{[u(t)]}\right) = 1$ and $h^B\left(\dt{[u(t)]}\right) = \alpha'' - \alpha'^2 + 1$,
we have
\[ \alpha'' - \alpha'^2 + 1 \geq \frac{1}{C}~,\]
which we can rewrite
\begin{equation} \label{eq:Alpha}  \alpha'' \geq \alpha'^2 - \left(1-\frac{1}{C}\right)~.\end{equation}

Assume by contradiction that 
\[\alpha'(t_0) > \sqrt{1-\frac{1}{C}}\]
for some $t_0 \in \R$. Inequality \eqref{eq:Alpha}, together with Gronwall's lemma, implies that $\alpha'$ is bigger for all $t\geq t_0$ than the function $f$ solution of the ordinary differential equation
\[f' = f^2 - \left(1-\frac{1}{C}\right)\]
with initial condition
\[f(t_0) = \alpha'(t_0)~.\]
It is a simple exercise to compute explicitly the function $f$. One finds
\[f(t) = \sqrt{1-\frac{1}{C}} \left ( \frac{1+ D e^{2\sqrt{1-\frac{1}{C}}(t-t_0)}}{1 - D e^{2\sqrt{1-\frac{1}{C}}(t-t_0)}}\right)~,\]
where
\[D = \frac{\alpha'(t_0) - \sqrt{1-\frac{1}{C}}}{\alpha'(t_0) + \sqrt{1-\frac{1}{C}}}~.\]
In particular, $f$ goes to $+\infty$ when $t$ goes to 
\[t_{max} = t_0 - \frac{1}{2\sqrt{1-\frac{1}{C}}} \log(D)~,\]
which implies that $\alpha'(t)$ blows up at some time $t \leq t_{max}$. This is absurd since $\alpha'$ is a continuous function on $\R$.

Similarly, if we had $\alpha'(t_0) < -\sqrt{1-\frac{1}{C}}$ for some $t_0$, we would obtain that $\alpha'$ goes to $-\infty$ at some time $t<t_0$. We conclude that
\[\alpha'^2 \leq 1-\frac{1}{C}~.\]
\end{proof}

We now prove a sharper version of lemma \ref{l:ComparaisonHilbertBlaschke}.

\begin{lemme} \label{l:SharpComparison}
Let $C$ be a constant greater or equal to $1$. If $h^B \geq \frac{1}{C} h^H$, then for all $x,y \in \Omega$,
\[d^B(x,y) \leq d^H(x,y) + \sqrt{1-\frac{1}{C}}~.\]
\end{lemme}
Since the hypothesis of this lemma is always satisfied for some constant $C$ by Benoist--Hulin's theorem, lemma \ref{l:ComparaisonHilbertBlaschke} will follow.

\begin{proof}
Fix $t_1 < t_2 \in \R$. Remark that the Blaschke distance between $[u(t_1)]$ and $[u(t_2)]$ is bounded above by the length with respect to the Blaschke metric of the path $\{[u(t)], t\in [t_1,t_2]\}$. We therefore have
\begin{eqnarray*}
d^B([u(t_1)], [u(t_2)]) & \leq & \int_{t_1}^{t_2} \sqrt{h^B\left(\dt{[u(t)]}, \dt{[u(t)]}\right)} \\
\ & \leq & (t_2-t_1) \sqrt{\frac{1}{t_2-t_1} \int_{t_1}^{t_2} h^B\left(\dt{[u(t)]}, \dt{[u(t)]}\right) } \\
\ & \leq & (t_2-t_1) \sqrt{\frac{1}{t_2-t_1} \int_{t_1}^{t_2} \alpha''(t) - \alpha'^2(t) +1 } \\
\ & \leq & (t_2-t_1) \sqrt{\frac{1}{t_2-t_1} \left( t_2 - t_1 + \alpha'(t_2) - \alpha'(t_1)\right) }~. \\
\end{eqnarray*}

Now, using proposition \ref{p:ControleAlpha}, we obtain
\begin{eqnarray*}
d^B([u(t_1)], [u(t_2)]) &\leq & (t_2-t_1) \sqrt{ 1 + 2\frac{\sqrt{1-\frac{1}{C}}}{t_2-t_1}} \\
\ & \leq & t_2-t_1 + \sqrt{1-\frac{1}{C}} = d^H([u(t_1)], [u(t_2)])+\sqrt{1-\frac{1}{C}}~.
\end{eqnarray*}

\begin{rmk}
The key fact is that, when integrating $h^B\left(\dt{[u(t)]}, \dt{[u(t)]}\right)$, the term in $\alpha''$, which is the only one that can make the Blaschke metric bigger than the Hilbert metric, contributes only up to a constant since its integral is just the difference of two values of $\alpha'$.
\end{rmk}

We have thus proved that
\[d^B(x,y) \leq d^H(x,y) + \sqrt{1-\frac{1}{C}} \]
for any two points on the projective segment joining $[\mathbf{e}_1]$ and $[\mathbf{e}_2]$. Since $[\mathbf{e}_1]$ and $[\mathbf{e}_2]$ where any two points in $\partial \Omega$, this concludes the proof of lemma \ref{l:SharpComparison} and thus lemma \ref{l:ComparaisonHilbertBlaschke}.
\end{proof}

\begin{rmk}
So far, we didn't use in an essential way that $h^B$ is the Blaschke metric. Lemma \ref{l:SharpComparison} is actually valid if we replace $h^B$ by any metric defined as the second fundamental form of some strictly convex hypersurface asymptotic to $\Omega$. However, we don't know any other metric to which it would be interesting to apply this lemma.
\end{rmk}

Lemma \ref{l:SharpComparison} is ``sharper'' than lemma \ref{l:ComparaisonHilbertBlaschke} in the sense that it relates explicitly the majoration of $h^B$ to its minoration. If we apply it to the particular case where $C = 1$, we obtain the following corollary:

\begin{MonCoro}
Let $\Omega$ be a proper convex open set in $\ProjR{n}$. Denote by $h^H$ (resp. $h^B$) its Hilbert metric. If
\[h^B \geq h^H\]
everywhere, then 
\[h^B = h^H\]
everywhere and $\Omega$ is an ellipsoid.
\end{MonCoro} 

\begin{proof}
By lemma \ref{l:SharpComparison}, if $h^B \geq h^H$, then $d^B \leq d^H$, which also implies that $h^B \leq h^H$, and thus $h^B= h^H$. In particular, the Hilbert metric is Riemannian, which implies that $\Omega$ is an ellipsoid (see section \ref{ss:HilbertMetric}).

\end{proof}

\subsection{Volume entropy of the Blaschke and Hilbert metrics}

Recall that the entropy of the Hilbert metric is defined by 
\[\Ent(h^H) = \limsup_{R\to + \infty} \frac{1}{R} \log \Vol \left( B^H(o,R)\right)~,\]
where $o$ is any base point in $\Omega$ and $\Vol$ is the volume with respect to any regular volume form $\vol$. According to Benoist--Hulin's theorem, one can take $\vol$ to be the volume form induced by the Blaschke metric.

\begin{lemme} \label{l:EntropyHilbertBlaschke}
\[\Ent(h^H) \leq \Ent(h^B)~.\]
\end{lemme}

\begin{proof}
For $R >0$, denote by $B^H(o,R)$ (resp. $B^B(o,R)$) the ball of center $o$ and radius $R$ with respect to the Hilbert (resp. Blaschke) metric. According to lemma \ref{l:ComparaisonHilbertBlaschke}, we have
\[B^H(o,R) \subset B^B(o,R+1)~.\]
Therefore
\[\Vol(B^H(o,R)) \leq \Vol(B^B(o,R+1))~,\]
and thus 
\begin{eqnarray*}
\Ent(h^H) & = & \limsup_{R \to + \infty} \frac{1}{R} \log \Vol\left(B^H(o,R)\right)\\
\ & = & \limsup_{R \to + \infty} \frac{1}{R+1} \log \Vol\left(B^H(o,R)\right)\\
\ & \leq & \limsup_{R \to + \infty} \frac{1}{R+1} \log \Vol\left(B^B(o,R+1)\right)\\
\ & \leq & \Ent(h^B)~.\\
\end{eqnarray*}
\end{proof}

In order to conclude the proof of theorem \ref{t:VolumeEntropy}, it is enough to prove:
\begin{proposition}
\[\Ent(h^B) \leq n-1~.\]
\end{proposition}

\begin{proof}
The proof is straightforward when we put together the theorem of Calabi stating that the Blaschke metric has Ricci curvature bounded below by $-(n-1)$ (theorem \ref{t:RicciBlaschke}) and a famous theorem of Bishop (see \cite{GallotHulinLafontaine}, theorem 3.101):
\begin{CiteThm}[Bishop] \label{t:Bishop}
If $g$ is a smooth Riemannian metric on a manifold $M$ whose Ricci curvature is bounded below by $-(n-1)$, then for all $R$,
\[\Vol_g(B_g(o,R)) \leq \Vol_{\H^n}(B_{\H^n}(o',R))~,\]
where $\Vol_{\H^n}(B_{\H^n}(o',R))$ is the volume of a ball of radius $R$ in the hyperbolic space $\H^n$ (with its metric of constant curvature $-1$).
\end{CiteThm}
It follows from Bishop's theorem that the entropy of the Blaschke metric is bounded above by the volume entropy of $\H^n$, which is well known to be~$n-1$.
\end{proof}

\section{The dimension $2$ case and surface group representations} \label{s:HitchinRepresentations}

In this section, we apply lemma \ref{l:ComparaisonHilbertBlaschke} to the study of the length spectrum of Hitchin representations in $\PSL(3,\R)$. We prove theorem \ref{t:MinorationHitchin} and describe the behaviour of this length spectrum ``far'' from the Fuchsian locus.\\

Let us recall first the definition we use of the translation length of an isometry.
\begin{definition}
The translation length of an isometry $g$ of a metric space $(X,d)$ is the number
\[l(g) = \lim_{n\to + \infty} \frac{1}{n} d(x,g^n\cdot x)~,\]
where $x$ is any point of $X$.
\end{definition}

If $\Omega$ is a proper convex open domain of $\ProjR{n}$ and $g$ a projective transformation such that $g(\Omega) = \Omega$, we will denote by $l^H(g)$ (resp. $l^B(g)$) the translation length of $g$ seen as an isometry of $\Omega$ with its Hilbert (resp. Blaschke) metric. As a consequence of lemma \ref{l:ComparaisonHilbertBlaschke}, we easily obtain the following corollary:

\begin{MonCoro} \label{c:TranslationHilbertBlaschke}
Let $\Omega$ be a proper convex open domain of $\ProjR{n}$ and $g$ a projective transformation such that $g(\Omega) = \Omega$. Then
\[l^B(g) \leq l^H(g)~.\]
\end{MonCoro}

\begin{proof}
By lemma \ref{l:ComparaisonHilbertBlaschke}, we have
\[\frac{1}{n} d^B(x,g^n\cdot x) \leq \frac{1}{n} d^H(x,g^n\cdot x) + \frac{1}{n}~.\]
Passing to the limit, we get
\[l^B(g) \leq l^H(g)~.\]
\end{proof}

Let us now specialize this to divisible convex sets in dimension $2$.

\subsection{Proof of theorem \ref{t:MinorationHitchin}}

Fix a closed connected oriented surface $S$ of genus greater than $1$. Denote by $\Gamma$ its fundamental group and consider
\[\rho: \Gamma \to \PSL(3,\R)~\]
a Hitchin representation. According to Choi-Goldman's theorem, $\rho(\Gamma)$ acts freely, properly discontinuously and cocompactly on a proper convex domain $\Omega_\rho \subset \ProjR{2}$.

The Blaschke metric $h^B$ on $\Omega_\rho$ is preserved by $\rho$ and thus induces a Riemannian metric on
\[\Omega_\rho/\rho(\Gamma)\simeq S\] that we still denote $h^B$. By Poincar\'e--Koebe's uniformization theorem, there exists a unique complete Riemannian metric $h^P$ on $\Omega_\rho$, conformal to $h^B$ and of constant curvature $-1$. Moreover, this metric is also invariant under the action of $\rho(\Gamma)$. We also denote $h^P$ the induced metric on $\Omega_\rho/ \rho(\Gamma)$.

\begin{lemme} \label{l:BlaschkePoincaré}
Either $h^B = h^P$ or there exists a constant $K>1$ such that
\[h^B \geq K h^P~.\]
\end{lemme}

\begin{proof}
Recall that, in this particular case, Calabi's theorem states that the Gauss curvature $\kappa^B$ of the Blaschke metric satisfies
\[-1 \leq \kappa^B \leq 0~.\]

We now use the following classical fact, sometimes refered to as the Ahlfors--Schwarz--Pick lemma. See \cite{Wolpert82} for a fairly general version.
\begin{lemme} \label{l:AhlforsSchwarzPick}
Let $h$ and $h'$ be two conformal metrics on a closed surface. If $\kappa(h) \leq \kappa(h') \leq 0$, then either  $h' = h$ everywhere or there is a constant $K>1$ such that $h' \geq K h$.
\end{lemme}
Applying this to $h^P$ and $h^B$ on $\Omega_\rho /\rho(\Gamma)$ gives lemma~\ref{l:BlaschkePoincaré}.
\end{proof}

Now, the convex $\Omega_\rho$ with the metric $h^P$ is locally isometric to the hyperbolic plane $\H^2$. Since $\rho$ is injective and acts properly discontinuously on $\Omega_\rho$, we can find a Fuchsian representation $j$ for which there is a $(\rho,j)$-equivariant isometry from $(\Omega_\rho, h^P)$ to $\H^2$.

Lemma \ref{l:BlaschkePoincaré}, together with corollary \ref{c:TranslationHilbertBlaschke}, implies the following:

\begin{coro}
Either $\rho$ is Fuchsian or there exists a constant $K>1$ such that
\[L_\rho \geq K L_j~.\]
\end{coro}

\begin{proof}
If $h^B = h^P$, then $\Omega_\rho$ is a disc (this can be deduced for instance from \cite{BenoistHulin14}) and $\rho$ is itself Fuchsian. Otherwise, there is a constant $K>1$ such that $h^B \geq K h^P$. Let $\gamma$ be any element of $\Gamma$. Denote by $l^P(\gamma)$ (resp. $l^B(\gamma)$, $l^H(\gamma)$) the translation length of $\rho(\gamma)$ with respect to the metric $h^P$ (resp. $h^B$, $h^H$). We have
\[l^B(\gamma) \geq K l^P(\gamma) = K L_j(\gamma)~.\]
On the other side, we have, by corollary \ref{c:TranslationHilbertBlaschke},
\[l^B(\gamma) \leq l^H(\gamma) = L_\rho(\gamma)~.\]
Thus
\[L_\rho \geq K L_j~.\]
\end{proof}

This concludes the proof of theorem \ref{t:MinorationHitchin}.

\subsection{Asymptotic behaviour of the length spectrum} \label{ss:AsymptoticSpectrum}

We end this section with a description of the asymptotic behaviour of the length spectrum away from the Fuchsian locus. The following results are consequences of the work of Loftin \cite{Loftin07} and Benoist--Hulin \cite{BenoistHulin13}. As Gilles Courtois pointed out to us, they are essentially contained in a recent paper by Xin Nie \cite{Nie15} (though Xin Nie focuses on the entropy of Hitchin representations in $\PSL(3,\R)$).\\

Let $J$ be a conformal structure on $S$ and $\Phi$ a holomorphic cubic differential on $(S,J)$. According to the theorem of Labourie and Loftin (theorem \ref{t:LabourieLoftin}), there exists, up to conjugation, a unique Hitchin representation $\rho_t$ such that $(S,J,t\Phi)$ identifies with the quotient by $\rho_t(\Gamma)$ of $\Omega_{\rho_t}$ with the conformal structure of its Blaschke metric and its Pick form. 

For any real $t$, denote by $h_t$ the Blaschke metric associated to the pair $(J,t\Phi)$ (seen as a conformal metric on $(S,J)$). Loftin proved in \cite{Loftin07} that $h_t$ goes to infinity proportionally to $t^{2/3}$ away from the zeros of $\Phi$. More precisely,

\begin{CiteThm}[Loftin, \cite{Loftin07}]
Let $S'$ be the complement of the zeros of $\Phi$ in $S$. Denote by $h_t$ the Blaschke metric corresponding to the pair $(J,t\Phi)$, by $h^P$ the conformal metric of curvature $-1$ on $(S,J)$ and by $\sigma_t$ the positive function such that
\[h_t = \sigma_t h^P~.\]
Then
\[\frac{\sigma_t}{t^{2/3}} \underset{t\to +\infty}{\to} 2^{1/3}|\Phi|^{2/3}\]
uniformly on every compact subset of $X'$. (Here, $|\Phi|$ is the pointwise norm of $\Phi$ with respect to the metric $h^P$.)
\end{CiteThm}

Using this result together with Benoist--Hulin's theorem, one easily deduces that the length spectrum of the Hitchin representation associated to $(J,t\Phi)$ grows uniformly like $t^{2/3}$.

\begin{coro}
Let $\rho_t$ be the Hitchin representation associated to $(J,t\Phi)$ and $j$ the Fuchsian representation uniformizing $(S,J)$. Then there is some $t_0 \geq 0$ and some constant $C>1$ such that
\[\frac{1}{C} t^{2/3} L_j \leq L_{\rho_t} \leq C t^{2/3} L_j\]
for $t\geq t_0$.
\end{coro}
Finally, by compacity, one can chose the constant $C$ uniformly on all pairs $(J,\Phi)$ where $(S,J)$ lives in a compact subset of the moduli space of Riemann surfaces homeomorphic to $S$ and $\Phi$ satisfies $\norm{\Phi}_J = 1$. Denote by $\rho(J,\Phi)$ the Hitchin representation associated to the pair $(J,\Phi)$. 

\begin{coro} \label{c:AsymptoticLengthSpectrum}
For any compact subset $K$ of the moduli space of Riemann surfaces homeomorphic to $S$, there exists some constant $C(K)$ such that for all pairs $(J,\Phi)$ with $(S,J)\in K$, 
\[\frac{1}{C(K)} \norm{\Phi}_J^{2/3} L_j \leq L_{\rho(J,\Phi)} \leq C(K) \norm{\Phi}_J^{2/3} L_j~.\]
\end{coro}

\begin{rmk}
In \cite{Zhang13}, Zhang constructs sequences of Hitchin representations whose entropy goes to $0$, though the translation lengths of some curves on the surface remain bounded. According to corollary \ref{c:AsymptoticLengthSpectrum} those sequences are associated to pairs $(J_n, \Phi_n)$ where $(S,J_n)$ leaves every compact subset of the moduli space. (Otherwise the whole spectrum of $\rho(J_n,\Phi_n)$ would go to infinity.)
\end{rmk}

\begin{rmk}
Loftin studied in \cite{Loftin04} the asymptotic behaviour of Hitchin representations associated to pairs $(J_n, \Phi_n)$ where $(S,J_n)$ leaves every compact subset of the moduli space. It is likely that his results would give a more precise description of the behaviour of the length spectrum on the whole Hitchin component.
\end{rmk}

\bibliographystyle{plain} 

\bibliography{biblio}

\end{document}